\documentclass[11pt,a4paper]{article}

\usepackage[leqno]{amsmath}
\usepackage{graphicx,latexsym,euscript,makeidx,color,bm}
\usepackage{epstopdf,subfigure}
\usepackage{amsmath,amsfonts,amssymb,amsthm,thmtools,mathtools,mathrsfs,enumerate}
\usepackage[colorlinks,linkcolor=blue,citecolor=red]{hyperref}

\usepackage{geometry}
\allowdisplaybreaks[4]
\def\dbR{\mathbb{R}}

  \def\cX{{\cal X}}  
  \def\cY{{\cal Y}}  
    
             \def\hb{\hbox}
\def\ms{\medskip}              
        \def\lan{\langle}    
   \def\ran{\rangle}    
         
\def\no{\noindent}          
         
\def\nn{\nonumber}         
\def\rf{\eqref}            
\def\cd{\cdot}             
     \def\({\Big (}       \def\ba{\begin{aligned}}
\def\les{\leqslant}      \def\){\Big )}       \def\ea{\end{aligned}}
\def\ges{\geqslant}      \def\[{\Big[}        \def\bel{\begin{equation}\label}
          \def\]{\Big]}        \def\ee{\end{equation}}
      \def\q{\quad}        
\def\h{\widehat}         \def\qq{\qquad}      
                                              
\def\a{\alpha}  \def\G{\Gamma}   \def\g{\gamma}     
\def\b{\beta}   \def\D{\Delta}   \def\d{\delta}        \def\p{\phi}
         
\def\f{\varphi}   \def\l{\lambda}        \def\e{\varepsilon}
\def\t{\tau}          

\newtheoremstyle{thry}
{}      
{}      
{\sl}   
{}      
{\bf}   
{.}     
{.5em}  
{}      
\theoremstyle{thry}

\newtheorem{theorem}{Theorem}[section]
\newtheorem{proposition}[theorem]{Proposition}

\newtheorem{lemma}[theorem]{Lemma}

\theoremstyle{definition}

\newtheorem{example}[theorem]{Example}

\theoremstyle{remark}
\newtheorem{remark}[theorem]{Remark}

\def\punct{}
\newtheoremstyle{dotless}{}{}{\rm}{}{\bf}{\punct}{.5em}{}
\theoremstyle{dotless}

\newtheorem{taggedassumptionx}{}
\newenvironment{taggedassumption}[1]
 {\renewcommand\thetaggedassumptionx{#1}\taggedassumptionx}
 {\endtaggedassumptionx}

\makeatletter
   
   \@addtoreset{equation}{section}
   \newcommand{\setword}[2]{%
   \phantomsection
   #1\def\@currentlabel{\unexpanded{#1}}\label{#2}%
   }
\makeatother

\begin{document}

\title{\bf Solvability of  Coupled Forward-Backward Volterra Integral Equations}

\author{
Wenyang Li\thanks{School of Mathematical Sciences, Shenzhen University, Shenzhen,
518060, China (Email: {\tt liwenyang2023@email.szu.edu.com}). }
~~~
Hanxiao Wang\thanks{Corresponding author. School of Mathematical Sciences, Shenzhen University, Shenzhen,
518060, China (Email: {\tt hxwang@szu.edu.cn}). This author is supported in part by NSFC Grant 12201424,
Guangdong Basic and Applied Basic Research Foundation 2023A1515012104,
and the Science and Technology Program of Shenzhen RCBS20231211090537064.}
~~~
Jiongmin Yong\thanks{Department of Mathematics, University of Central Florida, Orlando, FL 32816, USA
                 (Email: {\tt jiongmin.yong@ucf.edu}).
                 This author is supported by  NSF Grant DMS-2305475.}
}

\maketitle

\no\bf Abstract. \rm
Motivated by the optimality system associated with controlled (forward) Volterra integral equations (FVIEs, for short), the well-posedness of  coupled forward-backward Voterra integral equations (FBVIEs, for short) is studied. The main feature of FBVIEs is that the unknown $\{(\cX(t,s),\cY(t,s))\}$ has two arguments. By taking $t$ as a parameter and $s$ as a (time) variable, one can regard FBVIE as a system of ordinary differential equations (ODEs, for short), with infinite-dimensional space values $\{(\cX(\cd,s),\cY(\cd,s));\,s\in[0,T]\}$. To establish the well-posedness of such an FBVIE, a new  non-local monotonicity condition is introduced, by which a bridge
in infinite-dimensional spaces is constructed. Then by generalizing the method of continuation developed by \cite{Hu-Peng1995,Yong1997,Peng-Wu1999}
for differential equations, we have established the well-posedness of FBVIEs.
The key is to apply the chain rule to the mapping $t\mapsto\big[\int_\cd^T\lan \cY(s,s),\cX(s,\cd)\ran ds+\lan G(\cX(T,T)),\cX(T,\cd)\ran\big](t)$.

\ms

\no\bf Keywords. \rm
forward-backward Volterra integral equation, infinite-dimensional system of ordinary differential equations,
two-point boundary problem,  method of continuation.

\ms

\no\bf AMS subject classifications. \rm 45D05, 45J05, 49K21.

\section{Introduction}

In a number of problems involving biology systems, finance models, and fractional-order differential dynamics,
(forward) Volterra integral equations (FVIEs, for short) and the related optimal control theory have attracted strong attention recently.
The main feature is that FVIE can be used to describe some dynamics involving  memory.
Further, if we consider an optimal control problem of FVIE with a Bolza type cost functional, then by applying the Pontryagin maximum principle,
one will get a coupled forward-backward Volterra integral equation FBVIE, for short) (see \cite{Angell1976,Carlson1987,Yong2008,Lin2020}). Let us briefly look at such a problem. Consider the following state equation:
\bel{state}X(t)=x(t)+\int_0^tb(t,s,X(s),u(s))ds,\qq t\in[0,T],\ee
with the following cost functional
\bel{cost}J(u(\cd))=h(X(T))+\int_0^Tg(t,X(t),u(t))dt.\ee
In the above, $X(\cd)$ is the state valued in $\mathbb{R}^n$, and $u(\cd)$ is the control valued in $\dbR^m$. For simplicity, we assume that $(X,u)\mapsto(
b(t,s,X,u),h(X),g(t,X,u))$ is differentiable, and the control domain is the whole space $\dbR^m$. We assume that the state has no constraint (and therefore, the state space is the whole space $\dbR^n$). Our optimal control problem is to minimize \rf{cost}, subject to \rf{state}. Now, let $(\bar X(\cd),\bar u(\cd))$ be an optimal pair of this optimal control problem. For any admissible control $u(\cd)$, let $X(\cd)$ be the solution to the following variational equation:
\begin{align*}
X(t)&=\int_0^t\(b_x(t,s,\bar X(s),\bar u(s))X(s)+b_u(t,s,\bar X(s),\bar u(s))u(s)\)ds\\
&\equiv\f(t)+\int_0^tA(t,s)X(s)ds,\q t\in[0,T],
\end{align*}
where
$$
\f(t)=\int_0^tb_u(t,s,\bar X(s),\bar u(s))u(s)ds,\qq A(t,s)=b_x(t,s,\bar X(s),\bar u(s)),\q 0\les s\les t\les T.
$$
We let $\bar Y(\cd)$ be a continuous function satisfying the following adjoint equation:
$$\bar Y(t)=\psi(t)+\int_t^TA(s,t)^\top\bar Y(s)ds,\q t\in[0,T],$$
for some undetermined $\psi(\cd)$. Then
\begin{align*}
&\int_0^T\bar Y(t)^\top X(t)dt=\int_0^T\[\bar Y(t)^\top\(\f(t)+\int_0^tA(t,s)X(s)ds\)\]dt\\
&\q=\int_0^T\bar Y(t)^\top\f(t)dt+\int_0^T\(\int_s^TA(t,s)^\top \bar Y(t)dt\)^\top X(s)ds\\
&\q=\int_0^T\bar Y(t)^\top\f(t)dt+\int_0^T\(\bar Y(s)-\psi(s)\)^\top X(s)ds.
\end{align*}
Hence,
$$\int_0^T\psi(t)^\top X(t)dt=\int_0^T\bar Y(t)^\top\f(t)dt.$$
On the other hand, by the minimality of the optimal pair $(\bar X(\cd),\bar u(\cd))$, one has
\begin{align*}
0&\les \lim_{\e\downarrow 0}{J(\bar u(\cd)+\e u(\cd))-J(\bar u(\cd))\over\e}\\
&=h_x(\bar X(T))X(T)+\int_0^T\(g_x(s,\bar X(s),\bar u(s))X(s)+g_u(s,\bar X(s),\bar u(s))u(s)\)ds\\
&=h_x(\bar X(T))\int_0^T\(b_x(T,s,\bar X(s),\bar u(s))X(s)+b_u(T,s,\bar X(s),\bar u(s))u(s)\)ds\\
&\q+\int_0^T\(g_x(s,\bar X(s),\bar u(s))X(s)+g_u(s,\bar X(s),\bar u(s))u(s)\)ds\\
&=\int_0^T\(h_x(\bar X(T))b_x(T,s,\bar X(s),\bar u(s))+g_x(s,\bar X(s),\bar u(s))\)X(s)ds\\
&\q+\int_0^T\(h_x(\bar X(T))b_u(T,s,\bar X(s),\bar u(s))+g_u(s,\bar X(s),\bar u(s))\)u(s)ds.
\end{align*}
Now, by taking
$$\psi(s)=\(h_x(\bar X(T))b_x(T,s,\bar X(s),\bar u(s))+g_x(s,\bar X(s),\bar u(s))\)^\top,\q s\in[0,T],$$
the above gives
\begin{align*}
0&\les\int_0^T\bar Y(t)^\top\int_0^tb_u(t,s,\bar X(s),\bar u(s))u(s)dsdt\\
&\q+\int_0^T\(h_x(\bar X(T))b_u(T,s,\bar X(s),\bar u(s))+g_u(s,\bar X(s),\bar u(s))\)u(s)ds\\
&=\int_0^T\(\int_s^T\bar Y(t)^\top b_u(t,s,\bar X(s),\bar u(s))dt\\
&\q+h_x(\bar X(T))b_u(T,s,\bar X(s),\bar u(s))+g_u(s,\bar X(s),\bar u(s))\)u(s)ds.
\end{align*}
Then we have the following {\it optimality system:}
\bel{OS}\left\{\begin{aligned}
&\bar X(t)=x(t)+\int_0^tb(t,s,\bar X(s),\bar u(s))ds, \\
&\bar Y(t)=g_x(t,\bar X(t),\bar u(t))^\top+b_x(T,t,\bar X(t),\bar u(t))^\top h_x(\bar X(T))^\top\\
&\qq\q+\int_t^Tb_x (s,t,\bar X(t),\bar u(t))^\top\bar Y(s)ds,\\
& g_u(t,\bar X(t),\bar u(t))^\top+b_u(T,t,\bar X(t),\bar u(t))^\top h_x(\bar X(T))^\top\\
&\qq\qq+\int_t^Tb_u(s,t,\bar X(t),\bar u(t))^\top\bar Y(s)ds=0,\end{aligned}\right.\qq t\in[0,T].
\ee
This is a coupled FBVIE.
The coupling is given in the third relation, which is called a {\it stationarity condition}
or an {\it optimality condition}. The solution of \rf{OS} will provide a candidate for the optimal
trajectory of the corresponding control problems.
Therefore, the well-posedness of \rf{OS}  is important, at least for the optimal control theory of FVIEs.

\ms

A careful observation of the above shows that in the case that if we are considering a problem
with a linear state equation  and the cost being quadratic in control, i.e.,
\bel{LCONVEX}
\begin{aligned}
& b(t,s,X,u)=A(t,s)X+B(t,s)u, \q 0\les s\les t\les T,\\
& g(t,X,u)=Q(t,X)+\lan R(t)u,u\ran,\q h(X)=M(X),\q t\in[0,T],
\end{aligned}
\ee
then the optimality condition reads
\bel{LCONVEX1}
2R(t)\bar u(t)+ B(T,t)^\top M_x(\bar X(T))
+\int_t^TB(s,t)^\top\bar Y(s)ds=0,\q t\in[0,T].\ee
By assuming the existence of $R(\cd)^{-1}$, we end up with
\bel{OS-1}
\left\{\begin{aligned}
\bar X(t)= & x(t)+\int_0^t \[A(t,s)\bar X(s)-{1\over 2}B(t,s) R(s)^{-1}\int_s^T B(r,s)^\top\bar Y(r)dr\\
&\qq-{1\over 2} B(t,s)R(s)^{-1}B(T,s)^\top M_x(\bar X(T))\] ds, \\
\bar Y(t)= & Q_x(t,\bar X(t))+A(T,t)^\top M_x(\bar X(T))+\int_t^T A(s,t)^{\top}\bar Y(s)ds,\end{aligned}\right.
\q t\in[0,T].
\ee
Motivated by the above, we consider the following  FBVIE with general coefficients:
\bel{FBVIE-main}\left\{\begin{aligned}
&X(t)= x(t)+\int_0^t f\Big(t,s,X(s),Y(s),\int_s^T K(s,r) Y(r)dr,X(T)\Big) ds, \\
&Y(t)=h(t,X(t),X(T))+\int_t^T g(t,s,X(s), Y(s))ds,\end{aligned}\right.\q t\in[0,T],\ee
where  $x(\cd)$ is a given continuous function, and
$f(\cd)$, $g(\cd)$, $h(\cd)$, and $K(\cd)$ are suitable mappings.
We will focus on studying the well-posedness of FBVIE \rf{FBVIE-main}.

\ms

Another motivation for studying FBVIEs is the so-called time-inconsistent optimal control problems,
in which the optimality system is also a coupled FBVIE; see \cite{Hu-Jin-Zhou2012,Wang-Yong2021,Hamaguchi2021-1}, for example.
In addition, for the feature of involving memory, FBVIE has  potential applications in biology models
\cite{Gopalsamy1980,Kot2001,AIOmari-Gourley2003},
finance models \cite{Comte1998,ElEuch2018,ElEuch2019}, and
infinite-dimensional partial differential equations \cite{Viens-Zhang2019,Wang-Yong-Zhang2022,Wang-Yong-Zhou2023,Bondi2023,Bonesini2023}.
For example, we can regard the forward equation and the backward equation as an  evolution system
and   a utility, respectively.
When the  evolution system is affected by the  utility,
one will get a coupled FBVIE immediately. Such a phenomenon has appeared in many applications.
A typical example  is the  so-called large investor model in finance (see Cvitani\'{c} and Ma \cite{Cvitanic1996}).

\ms

To recover the semi-group property of \rf{FBVIE-main},
inspired by \cite{Viens-Zhang2019,Wang-Yong-Zhang2022},
we introduce the following auxiliary system of ODEs with the unknown $(\cX(\cd,\cd),\cY(\cd,\cd))$
having two arguments:
\bel{FBVIE-main1}\left\{\begin{aligned}
&\cX_s(t,s)=  f\Big(t,s,\cX(s,s),\cY(s,s), \int_s^TK(s,r)\cY(r,r)dr,\cX(T,T)\Big),\\
& \qq\qq\qq\qq (t,s)\in\D_*[0,T],\\
&\cY_s(t,s)=  -g(t,s,\cX(s,s),\cY(s,s)), \q (t,s)\in\D^*[0,T],\\
&\cX(t,0)=x(t),\q \cY(t,T)=h(t,\cX(t,t),\cX(T,T)),\q t\in[0,T],
\end{aligned}\right.\ee
where $\D_*[0,T]=\{(t,s)\in[0,T]^2\hbox{ with } t\ges s\}$
and $\D^*[0,T]=\{(t,s)\in[0,T]^2\hbox{ with } t\les s\}$.
Indeed, \rf{FBVIE-main1} provides an equivalent representation of \rf{FBVIE-main} with the relationship:
\bel{FBVIE-ODEs}
X(t)=\cX(t,t),\q Y(t)=\cY(t,t),\q t\in[0,T].
\ee

\ms

For the forward-backward structure,
\rf{FBVIE-main1} is essentially a Fredholm-type integral equation,
and thus one cannot use the contraction mapping theorem unless $T>0$ is small enough.
When the coefficients do not depend on $t$, \rf{FBVIE-main} reduces to the following forward-backward
ODE (FBDE, for short):
\bel{FBDE}\left\{\begin{aligned}
&p(t)= x+\int_0^t a(s,p(s),q(s)) ds, \\
&q(t)=\psi(p(T))+\int_t^T b(s,p(s), q(s))ds,\end{aligned}\right.\q t\in[0,T].\ee
For FBDEs, which can be regarded as a special case of FBSDEs, two types of methods have been developed to prove  solvability.
The first one is the so-called {\it four-step method}  (also called a  {\it decoupling method}),
which was initiated by Ma, Protter, and Yong \cite{Ma-Protter-Yong1994}.
The decoupling method mainly depends on the theory of partial differential equations (PDEs, for short).
Since the PDE associated with \rf{FBVIE-main1} is defined on the path space, which is infinite-dimensional,
its solvability might be more difficult than that of \rf{FBVIE-main1} itself; see Wang, Yong, and Zhang \cite{Wang-Yong-Zhang2022}, for example.
Thus, it seems too difficult to extend the decoupling method of FBDEs   to FBVIE \rf{FBVIE-main1}.

\ms
The second one is called a {\it method of  continuation}
(or a {\it monotonicity method}), which was introduced by Hu and Peng \cite{Hu-Peng1995},
and then developed by \cite{Yong1997,Peng-Wu1999}.
Some further developments on FBSDEs/FBDEs can be found in \cite{Delarue2002,Yong2010,Ma2015,Yu2022}.
The main idea of continuation method  is  to reach the solvability of FBDEs
by  starting with a known solvable FBDE.
The way  is to apply  the chain rule  together with some sort of monotonicity  conditions
to get certain a priori estimates. More precisely, in Hu and Peng \cite{Hu-Peng1995},
they  introduced the following monotonicity  condition:
\bel{FBDE-MC}\begin{aligned}
&\lan a(s,p_1,q_1)-a(s,p_2,q_2),\, q_1-q_2\ran-\lan b(s,p_1,q_1)-b(s,p_2,q_2),\, p_1-p_2\ran\\
&\q\les-\a|p_1-p_2|^2-\b|q_1-q_2|^2,\q \forall p_1,p_2,q_1,q_2\in\dbR^n,\\
&\lan \psi(p_1)-\psi(p_2),\, p_1-p_2\ran\ges \g |p_1-p_2|^2,\q \forall p_1,p_2\in\dbR^n,\end{aligned}\ee
where $\a,\b,\g>0$ are given constants, and then applied the chain rule to the following function:
\bel{FBDE-Ito}
\lan p(t),q(t)\ran,\q t\in[0,T].
\ee
In this paper, we will develop this method to prove the solvability of FBVIE \rf{FBVIE-main}. However, it is by no means easy;
some comments and discussions can be found in our previous paper \cite[Subsection 5.1]{Wang-Yong-Zhang2022}.

\ms
We now  make a careful analysis of FBVIE \rf{FBVIE-main}.
Recall that \rf{FBVIE-main} is equivalent to \rf{FBVIE-main1}.
From now on, we will focus on the auxiliary system \rf{FBVIE-main1} rather than  \rf{FBVIE-main}.
Since  at the point $(t,s)$, \rf{FBVIE-main1} involves the values $\cX(s,s)$ and $\cY(s,s)$ at $(s,s)$,
it is a non-local system.
Alternatively, we can view  $t$  as a parameter, and then \rf{FBVIE-main1}
can be regarded as a Hamiltonian system of  ODEs
with the solution $\{(\cX(\cd,s),\cY(\cd,s)\}_{s\in[0,T]}$,
and the two-point boundary condition $\cX(\cd,0)=x(\cd),\,\, \cY(\cd,T)=h(\cd,\cX(\cd,\cd),\cX(T,T))$.
Thus, \rf{FBVIE-main1} can be regarded as  an  infinite-dimensional version of FBDE \rf{FBDE}.

\ms
Note that  the monotonicity condition \rf{FBDE-MC} is defined at a local point $(t,x,y)$,
but \rf{FBVIE-main1} is a non-local  system  (or an infinite-dimensional system).
Then one immediately realizes that \rf{FBDE-MC} does not work to the system \rf{FBVIE-main1},
for which one should not apply the chain rule to the function $\lan \cX(t,s),\cY(t,s)\ran$ either.
%
So two questions appear naturally:

\ms
(i) \emph{How can one introduce a proper non-local monotonicity condition for \rf{FBVIE-main1}?}

\ms
(ii) \emph{Which function should we apply the chain rule?}

\ms

Our answer is  to impose the following {\it non-local monotonicity condition}:
\bel{MC-1}
\begin{aligned}
&\int_t^T \big\lan \h y(s),\, \h f(s,t)\big\ran ds-\Big\lan \h h(t)+\int_t^T\h g(t,s)ds,\, \h x(t) \Big\ran\\
&+\big\lan G(x_1(T))- G(x_2(T)), \,\h f(T,t)\big\ran \les -\gamma|\h x(t)|^2,\q \hb{for some constant }\g>0,
\end{aligned}\ee
and some smooth function $G(\cd)$ satisfying
\bel{MC-3}
\begin{aligned}
&\lan G(x_1(T))- G(x_2(T)),\, x_1(T)-x_2(T)\ran\ges | G^{1\over 2}_0 [x_1(T)- x_2(T)]|^2,\\
& |G(x_1(T))- G(x_2(T))|\les K| G^{1\over 2}_0 [x_1(T)- x_2(T)]|,
\end{aligned}
\ee
where $G_0\ges 0$ is a positive semi-definite matrix,
$K>0$ is a constant,
and
\bel{MC-2}
\begin{aligned}
&\h x(t)=x_1(t)-x_2(t),\q \h y(t)=y_1(t)-y_2(t),\\
&\h h(t)=h(t,x_1(t), x_1(T))-h(t,x_2(t),x_2(T)),\\
&\h g(t,s)= g(t,s,x_1(s), y_1(s))- g(t,s,x_2(s), y_2(s)),\\
&\h f(t,s)=f\Big(t,s,x_1(s),y_1(s), \int_s^T K(s,r)y_1(r)dr,x_1(T)\Big)\\
&\qq\qq-f\Big(t,s,x_1(s),y_1(s), \int_s^T K(s,r)y_1(r)dr,x_2(T)\Big),
\end{aligned}\ee
for any $x_i(\cd),y_i(\cd)\in C([0,T];\dbR^n);\,i=1,2$.
Note that  \rf{MC-1} is a non-local monotonicity condition (or a monotonicity condition defined
on the infinite-dimensional space $C([0,T];\dbR^n)$).
Then, with the  method of continuation, by applying the chain rule to the following function:
\bel{FBVIE-Ito}
\int_t^T \lan \cX(s,t),\cY(s,s)\ran ds+\lan G( \cX(T,T)),\cX(T,t)\ran,\q t\in[0,T],
\ee
we can establish the well-posedness of \rf{FBVIE-main1}.

\ms

The main difficulty is to find an appropriate monotonicity condition for \rf{FBVIE-main1}.
This is achieved based on some new observations for FBDEs.
From Subsection \ref{subsec:comparison},
we can see that \rf{MC-1} is essentially an infinite-dimensional
version of \rf{FBDE-MC}, and when  FBVIE \rf{FBVIE-main1} reduces to an FBDE,
it is equivalent to \rf{FBDE-MC}. In Section \ref{sec:application}, we will show the standard condition in
LQ optimal control problems for FVIEs is sufficient for  \rf{MC-1}.
With the non-local monotonicity condition \rf{MC-1}, we still need to find a simple and known solvable
FBVIE as the start of applying the  method of continuation.
Interestingly, we will see that this known solvable FBVIE \rf{linear} is essentially an FBDE.

\ms

The rest of the paper is organized as follows.
In Section \ref{sec:pre}, we introduce some notations and state the main result of our paper.
The proof is given in Section \ref{sec:well}. More precisely,
we show the idea of finding the non-local monotonicity condition \rf{MC-1} in Subsection \ref{subsec:mc},
compare our non-local monotonicity condition  \rf{MC-1} with the local one \rf{FBDE-MC} in Subsection \ref{subsec:comparison},
prove the uniqueness result in Subsection \ref{subsec:unique}, and establish the solvability in Subsection \ref{subsec:existence}.
Finally, two simple examples are given in Section \ref{sec:application}.

\section{Preliminary and the main result}\label{sec:pre}

Let $T>0$ be a given time horizon and denote
$$
\D_*[0,T]=\{(t,s)\in[0,T]^2\hbox{ with } t\ges s\}
\hbox{ and } \D^*[0,T]=\{(t,s)\in[0,T]^2\hbox{ with } t\les s\}.
$$
We introduce the following spaces of functions:
\begin{align*}
& C([0,T];\dbR^n):
   \hb{~~the space of $\dbR^n$-valued, continuous functions on $[0,T]$};\\
& L^2([0,T];\dbR^n):
   \hb{~~the space of $\dbR^n$-valued, measurable, square integrable  functions on $[0,T]$}.
\end{align*}
Similarly, we can define the spaces of continuous functions on $\D_*[0,T]$
and $\D^*[0,T]$ as $C(\D_*[0,T];\dbR^n)$ and $C(\D^*[0,T];\dbR^n)$, respectively.

\ms

We impose the following assumptions for FBVIE \rf{FBVIE-main1}.

\begin{taggedassumption}{(A1)}\label{ass:A1}
The non-local monotonicity condition \rf{MC-1}--\rf{MC-3} holds.
\end{taggedassumption}
\begin{taggedassumption}{(A2)}\label{ass:A2}
The coefficients of FBVIE \rf{FBVIE-main1} are continuous functions.
Moreover, there exists a constant $L>0$ such that
\begin{align*}
&|f(t_1,s,x_1,y_1,y_1',x_1')-f( t_2,s, x_2,y_2,y_2',x_2')|\\
&\q\les L\big(|t_1- t_2|^\a+|x_1-x_2|+|y_1-y_2|+|G_0( x_1'-x_2')|+|y_1'-y_2'| \big), \\
&\qq \forall (t_i,s,x_i,y_i,y_i',x_i')\in\D_*[0,T]\times(\dbR^n)^4,\q i=1,2,\\
&|g(t_1,s,x_1,y_1)-g(t_2, s,x_2,y_2)|+|h(t_1,x_1,x_1')-h(t_2,x_2,x_2')|\\
&\q\les
L\big(|t_1- t_2|^\a +|x_1- x_2|+|y_1- y_2|+|G_0( x_1'-x_2')|\big), \\
&\qq \forall (t_i,s,x_i,y_i,x_i')\in\D^*[0,T]\times(\dbR^n)^3,\q i=1,2,
\end{align*}
for some $\a\in(0,1]$, where $G_0\ges 0$ is given by \rf{MC-3}.
\end{taggedassumption}

The main result of our paper can be stated as follows.

\begin{theorem}\label{Thm:well-posedness}
Let {\rm \ref{ass:A1}} and {\rm \ref{ass:A2}} hold.
Then the FBVIE \rf{FBVIE-main1} admits a unique solution $(\cX(\cd,\cd),\cY(\cd,\cd))
\in C(\D_*[0,T];\dbR^n)\times C(\D^*[0,T];\dbR^n)$.
\end{theorem}

\begin{remark}
The solvability of FBVIEs in a small time horizon was proved by Hamaguchi \cite{Hamaguchi2021}
by the classical fixed point theorem.
The case with an  arbitrarily given $T>0$ is much more challenging.
Inspired by \cite{Wang-Yong-Zhang2022}, the first solvability result of FBVIEs
in an arbitrarily long time horizon was obtained by Wang, Yong, and Zhou \cite{Wang-Yong-Zhou2023} for linear systems,
by introducing the so-called path-dependent Riccati equation. Essentially, this is a four-step method,
and heavily depends on the symmetric structure of the linear system.
To the best of our knowledge, \autoref{Thm:well-posedness} is the first result on the well-posedness of FBVIEs
with general nonlinear coefficients and an  arbitrarily given time horizon.
\end{remark}

\begin{remark}
In \cite{Hamaguchi2021,Wang-Yong-Zhang2022,Wang-Yong-Zhou2023}, the authors considered the stochastic case of FBVIEs,
which is a generalization of \rf{FBVIE-main1} with an additional It\^{o} integral.
We remark that the monotonicity method introduced in the current paper still works for stochastic FBVIEs.
Additionally, we need to handle some technical issues.
We hope to focus on introducing the basic idea here,
and will report the results of stochastic FBVIEs in future publications.
\end{remark}

\section{Well-posedness}\label{sec:well}

\subsection{Find a monotonicity condition  in infinite-dimensional space}\label{subsec:mc}

It is known that the monotonicity condition plays a central role in establishing the solvability of FBDEs
by the method of  continuation; see \cite{Hu-Peng1995,Yong1997,Peng-Wu1999}, for example.
However, since FBVIE \rf{FBVIE-main1} is a non-local system, the local monotonicity condition \rf{FBDE-MC} defined at a point
cannot be applied here. We need to find an infinite-dimensional version of it,
and then  construct  a bridge which can link two infinite-dimensional spaces.

\ms

The key is to well understand the intuition of calculating $d[\lan p(t),q (t)\ran]$  in
Hu and Peng \cite{Hu-Peng1995}.
Our idea is to return to the linear-quadratic optimal control theory.

\ms

\begin{center}
\textbf{An insight from value function}
\end{center}

Consider the controlled  linear ODE:
\bel{LQ-ODE1}
p(t)=x+\int_0^t\[A(s) p(s)+B(s) u(s)\]ds,\q t \in[0,T],
\ee
and the cost functional of a quadratic form:
\bel{LQ-ODE2}
J(u(\cd))=\int_0^T\[\lan Q(t)p(t),p(t)\ran+\lan R(t) u(t),u(t)\ran\] d t+\lan G p(T),p(T)\ran.
\ee
Then the corresponding Hamiltonian system reads
\bel{FBDE-LQ}\left\{\begin{aligned}
&p(t)= x+\int_0^t \big[A(s)p(s)-B(s)R(s)^{-1}B(s)^\top q(s)\big] ds, \q t\in[0,T],\\
& q(t)=G p(T)+\int_t^T \big[A(s)^\top q(s) +Q(s)p(s)\big] ds,\q t\in[0,T],\end{aligned}\right.\ee
with the optimal control
$$\bar u(s)=-R(s)^{-1}B(s)^\top q(s),\q s\in[0,T].$$
From Yong and Zhou \cite{Yong-Zhou1999}, one immediately finds that  \rf{FBDE-MC} almost coincides with the so-called standard condition,
but more importantly, we realize that $\lan p(t),q(t)\ran$ is indeed the function of optimal values; that is
$$\lan p(t),q(t)\ran=\int_t^T \[\lan Q(s)p(s),p(s)\ran+\lan R(s)\bar u(s),\bar u(s)\ran\]ds+\lan Gp(T),p(T)\ran.$$
Thus, essentially, Hu and Peng \cite{Hu-Peng1995} calculated the derivative of the optimal value with respect to the time variable,
and we know that the  derivative is always non-negative under proper conditions.
Based on this observation, our approach is beginning to come out:
We should calculate the optimal value of  optimal control problems for VIEs (i.e., \rf{FBVIE-Ito}),
by which we can find a proper monotonicity condition in infinite-dimensional spaces (i.e., \rf{MC-1}).

\subsection{Comparison between the non-local monotonicity condition \rf{MC-1} and Hu and Peng's condition \rf{FBDE-MC} }\label{subsec:comparison}

In this subsection, we will show that \rf{MC-1} is essentially
an infinite-dimensional version of Hu and Peng's
monotonicity condition  \rf{FBDE-MC} given in \cite{Hu-Peng1995}.
When the system \rf{FBVIE-main} reduces to an FBDE,
\rf{MC-1} is almost equivalent to  \rf{FBDE-MC}.

\ms

For simplicity, we only  consider the Hamiltonian system associated with the LQ optimal control problem \rf{LQ-ODE1}--\rf{LQ-ODE2}.
By taking  \rf{LQ-ODE1}--\rf{LQ-ODE2} as an LQ optimal control problem for VIEs, the associated FBVIE reads
$$
\left\{\begin{aligned}
X(t)= & x+\int_0^t \[A(s) X(s)-B(s) R(s)^{-1}\int_s^TB(s)^\top  Y(r)dr\\
&\qq-B(s)R(s)^{-1}B(s)^\top G X(T)\] ds, \\
Y(t)= & Q(t)X(t)+A(t)^\top G X(T)+\int_t^T A(t)^{\top} Y(s)ds,\end{aligned}\right.
\q t\in [0,T].
$$
Let
$$
p(t)=X(t),\q q(t)=\int_t^T Y(s)ds+GX(T),\q t\in[0,T],
$$
then we have
$$
\left\{\begin{aligned}
&p(t)=  x+\int_0^t \[A(s) p(s)-B(s) R(s)^{-1}B(s)^\top q(s) \]ds, \\
&q(t)= G p(T)+\int_t^T \[Q(s)p(s)+A(s)^\top q(s)\]ds,\end{aligned}\right.\qq t\in[0,T],
$$
which is an FBDE.
By \rf{MC-2} and \rf{FBDE-LQ}, the corresponding functions $\h x(\cd),\h y(\cd)$, $\h f(\cd), \h g(\cd), \h h(\cd)$,
and $a(\cd), b(\cd), \psi(\cd)$ can be well-defined.
Then
\begin{align*}
&\int_t^T \big\lan \h y(s),\, \h f(s,t)\big\ran ds-\Big\lan \h h(t)+\int_t^T\h g(t,s)ds,\, \h x(t) \Big\ran+\big\lan G\h x(T), \,\h f(T,t)\big\ran \\
&\q=\Big\lan \int_t^T  \h y(s)ds+G\h x(T),\, \h f(t) \Big\ran -\Big\lan \h h(t)+\int_t^T\h g(s)ds,\, \h x(t) \Big\ran\\
&\q=\big\lan q_1(t)-q_2(t),\,  a(t,p_1(t),q_1(t))-a(t,p_2(t),q_2(t))\big \ran\\
&\qq -\big\lan b(t,p_1(t),q_1(t))-b(t,p_2(t),q_2(t)),\,  p_1(t)-p_2(t) \big \ran,\q t\in[0,T].\end{align*}
Note that the above only depends on the values at a local point $t$.
Then from \rf{MC-1}, we have
\bel{MC-MC}\begin{aligned}
&\big\lan a(t,p_1,q_1)-a(t,p_2,q_2),\,q_1-q_2\big \ran-\big\lan b(t,p_1,q_1)-b(t,p_2,q_2),\,  p_1-p_2 \big \ran\\
&\q \les -\g |p_1-p_2|^2,\q \forall p_1,p_2,q_1,q_2\in\dbR^n,\\
&\lan \psi(p_1)-\psi(p_2),\, p_1-p_2\ran=G|p_1-p_2|^2\ges 0,\q \forall p_1,p_2\in\dbR^n.
\end{aligned}\ee
The non-local monotonicity condition \rf{MC-1} reduces to a local one.
Clearly, the above monotonicity condition is almost equivalent to \rf{FBDE-MC}.
Indeed, \rf{MC-MC} is a little bit weaker than \rf{FBDE-MC},
and in \cite{Peng-Wu1999}, Peng and Wu proved the well-posedness of \rf{FBDE} under  a condition like this.

\subsection{Uniqueness}\label{subsec:unique}

\begin{proposition}\label{prop:uniqueness}
Let {\rm \ref{ass:A1}} and {\rm \ref{ass:A2}} hold. Then FBVIE \rf{FBVIE-main1} admits at most one solution.
\end{proposition}

\begin{proof}
Suppose that $(\cX_i(\cd,\cd),\cY_i(\cd,\cd))\in
C(\D_*[0,T];\dbR^n)\times  C(\D^*[0,T];\dbR^n);i=1,2$ satisfy \rf{FBVIE-main1}.
Denote
\begin{align}
&\h \cX(t,s)=\cX_1(t,s)-\cX_2(t,s),\q  \h \cY(t,s)=\cY_1(t,s)-\cY_2(t,s), \nonumber\\
 &\h  h(t)=h(t,\cX_1(t,t),\cX_1(T,T))-h(t,\cX_2(t,t),\cX_2(T,T)),\nn\\
 & \h g(t,s)= g(t,s,\cX_1(s,s),\cY_1(s,s))- g(t,s,\cX_2(s,s),\cY_2(s,s)),\nn\\
 &\h f(t,s)=f\Big(t,s,\cX_1(s,s),\cY_1(s,s), \int_s^T K(s,r)\cY_1(r,r)dr,\cX_1(T,T)\Big)\nn\\
 &\qq\qq-f\Big(t,s,\cX_2(s,s),\cY_2(s,s), \int_s^T K(s,r)\cY_2(r,r)dr,\cX_2(T,T)\Big).\nn
\end{align}
Then
$$
\left\{\begin{aligned}
&{d\h\cX(t,s)\over ds}= \h f(t,s), \q (t,s)\in\D_*[0,T],\\
&{d \h\cY(t,s)\over ds}= -\h g(t,s), \q (t,s)\in\D^*[0,T],\\
&\h\cX(t,0)=0,\q \h\cY(t,T)=\h h(t),\q t\in[0,T],
\end{aligned}\right.
$$
which implies that
\begin{align*}
&{d\over dt}\[\int_t^T \lan \h\cY(s,s),\h\cX(s,t)\ran ds+\lan G(\cX_1(T,T)) -G(\cX_2(T,T)),\,\h\cX(T,t)\ran\]\\
&\q=-\lan \h\cY(t,t),\h\cX(t,t)\ran+\int_t^T\lan \h\cY(s,s),\h\cX_t(s, t)\ran ds\\
&\qq+\lan  G(\cX_1(T,T)) -G(\cX_2(T,T)),\h\cX_t(T, t)\ran\\
&\q=-\Big\lan\h h(t)+\int_t^T \h g(t,s)ds,\, \h\cX(t,t)\Big\ran+\int_t^T \lan \h\cY(s,s),\,\h f(s,t)\ran ds\\
&\qq+\big\lan G(\cX_1(T,T)) -G(\cX_2(T,T)), \,\h f(T,t)\big\ran\\
&\q\les-\g |\h\cX(t,t)|^2,\q t\in[0,T].
\end{align*}
Thus,
\begin{align*}
&\lan G(\cX_1(T,T)) -G(\cX_2(T,T)),\,\h\cX(T,T)\ran+\g\int_0^T  |\h\cX(t,t)|^2dt\les0.
\end{align*}
Note that
$$
\lan G(\cX_1(T,T)) -G(\cX_2(T,T)),\,\h\cX(T,T)\ran\ges| G^{1\over 2}_0 [\cX_1(T,T)- \cX_2(T,T)]|^2\ges0.
$$
It turns out that
$$
\cX_1(t,t)=\cX_2(t,t),\q t\in[0,T].
$$
Then, both $(\cX_1(\cd,\cd),\cY_1(\cd,\cd))$ and $(\cX_2(\cd,\cd),\cY_2(\cd,\cd))$ satisfy
\bel{uniqueness-proof1}
\left\{\begin{aligned}
&{d \cX_i(t,s)\over ds}=  f\Big(t,s,\cX(s,s),\cY_i(s,s), \int_s^T K(s,r)\cY_i(r,r)dr,\cX(T,T)\Big), \\
&\qq\qq\qq (t,s)\in\D_*[0,T],\\
&{d \cY_i(t,s)\over ds}= - g(t,s,\cX(s,s),\cY_i(s,s)), \q (t,s)\in\D^*[0,T],\\
&\cX_i(t,0)=x(t),\q \cY_i(t,T)=h(t,\cX(t,t),\cX(T,T)),\q t\in[0,T],
\end{aligned}\right.
\ee
where
$$
\cX(t,t):=\cX_1(t,t)=\cX_2(t,t),\q t\in[0,T].
$$
Clearly, FBVIE \rf{uniqueness-proof1} is a decoupled one,
and then admits a unique solution. Thus,
$(\cX_1(\cd,\cd),\cY_1(\cd,\cd))=(\cX_2(\cd,\cd),\cY_2(\cd,\cd))$.
\end{proof}

\subsection{Existence}\label{subsec:existence}

We now give an existence result of FBVIE \rf{FBVIE-main1}.
We begin with a simple FBVIE, whose solvability can be obtained
by modifying the result in Peng and Wu \cite{Peng-Wu1999}.

\begin{lemma}\label{lemm:alpha=0}
For any given continuous functions $f_0(\cd,\cd)$, $g_0(\cd,\cd)$, $h_0(\cd)$, and $x(\cd)$,
the following linear FBVIE
\bel{linear}\left\{\ba
&\cX_s(t,s)=- \int_s^T \cY(r,r)dr-G(\cX(T,T))+f_0(t,s),\q (t,s)\in\D_*[0,T], \\
&\cY_s(t,s)=-g_0(t,s),\q (t,s)\in\D^*[0,T],\\
&\cX(t,0)=x(t),\q \cY(t,T)=\cX(t,t)+h_0(t),\q t\in[0,T],\ea\right.\ee
admits a unique solution $(\cX(\cd,\cd),\cY(\cd,\cd))
\in C(\D_*[0,T];\dbR^n)\times C(\D^*[0,T];\dbR^n)$,
where $G(\cd)$ is given by the monotonicity condition \rf{MC-3}.
\end{lemma}

\begin{proof}
One can easily check that \rf{linear} satisfies the non-local monotonicity condition \rf{MC-1}.
Thus, by \autoref{prop:uniqueness}, it admits at most one solution.
For the  existence of a solution to \rf{linear}, we divide the proof into two steps.

\ms

\emph{Step 1.} We first consider the case with the smooth
coefficient $f_0(\cd,\cd)$ and
smooth initial value $x(\cd)$. By Peng and Wu \cite{Peng-Wu1999},
we know that the following FBDE admits a unique solution $(X(\cd),Y(\cd))$:
\bel{lemm:alpha=0-p1}\left\{\ba
&dX(t)=\[-Y(t)+f_0(t,t)+x^{\prime}(t)+\int_0^t \frac{\partial f_0(t,r)}{\partial t}dr\]dt, \q t\in[0,T],\\
&-dY(t)=\[X(t)+h_0(t)+\int_t^T g_0(t,r)dr\]dt,\q t\in[0,T],\\
&X(0)=x(0),\q Y(T)=G(X(T)).\ea\right.
\ee
Let $\cX(t,t)=X(t)$ and $\cY(t,t)=-{d Y(t)\over dt}$, then from the above we have
\begin{align*}
&\cX(t,t)=x(0)+\int_0^t \[-Y(s)+f_0(s,s)+x^{\prime}(s)+\int_0^s\frac{\partial f_0(s,r)}{\partial s}dr\]ds\\
&\q=x(t)-\int_0^t  Y(s)ds+\int_0^t f_0(s,s)ds+\int_0^t\int_r^t\frac{\partial f_0(s,r)}{\partial s}dsdr\\
&\q=x(t)+\int_0^t f_0(t,s)ds-\int_0^t \[ Y(T)-\int _s^T{dY(r)\over dr}dr\] ds\\
&\q=x(t)+\int_0^t  \[ f_0(t,s)-G(\cX(T,T))-\int _s^T\cY(r,r)dr\] ds,\q t\in[0,T],
\end{align*}
and
$$\cY(t,t)=\cX(t,t)+h_0(t)+\int_t^T g_0(t,r)dr,\q t\in[0,T].$$
Further, we define
\begin{align*}
\cX(t,s)&=x(t)+\int_0^s  \[ f_0(t,r)-G(\cX(T,T))-\int _r^T\cY(\t,\t)d\t\] dr,\q (t,s)\in\D_*[0,T],\\
\cY(t,s)&=\cX(t,t)+h_0(t)+\int_s^T g_0(t,r)dr,\q (t,s)\in\D^*[0,T],
\end{align*}
which, clearly, is a solution to  \rf{linear}.

\ms

\emph{Step 2.} We next consider the case with the
coefficient $f_0(\cd,\cd)$ and initial value $x(\cd)$ being continuous.
By the standard mollification method, we can find a sequence of smooth functions
$\{f_{0,n}(\cd,\cd),x_n(\cd)\}_{n>0}$ to uniformly converge to $f_0(\cd,\cd)$ and $x(\cd)$.
By the results obtained in Step 1, we know that the following equation
admits a unique solution $(\cX_n(\cd,\cd),\cY_n(\cd,\cd))$ for any given $n>0$:
\bel{X-m,n}\left\{\ba
&{d \cX_n(t,s)\over ds}=- \int_s^T \cY_n(r,r)dr-G(\cX_n(T,T))+f_{0,n}(t,s),\q (t,s)\in\D_*[0,T], \\
&{d\cY_n(t,s)\over ds}=-g_0(t,s),\q (t,s)\in\D^*[0,T],\\
&\cX_n(t,0)=x_n(t),\q \cY_n(t,T)=\cX_n(t,t)+h_0(t),\q t\in[0,T].\ea\right.
\ee
Denote $\cX_{m,n}(\cd,\cd)=\cX_m(\cd,\cd)-\cX_n(\cd,\cd)$,
$\cY_{m,n}(\cd,\cd)=\cY_m(\cd,\cd)-\cY_n(\cd,\cd)$,
$x_{m,n}(\cd)=x_m(\cd)-x_n(\cd)$,  $f_{0,m,n}(\cd,\cd)=f_{0,m}(\cd,\cd)-f_{0,n}(\cd,\cd)$ and
$\h G(\cX_{m,n}(T,T))=G(\cX_{m}(T,T))-G(\cX_{n}(T,T))$ for any given $m,n>0$.
Then,
$$\left\{\ba
&{d \cX_{m,n}(t,s)\over ds}=- \int_s^T \cY_{m,n}(r,r)dr-\h G(\cX_{m,n}(T,T))+f_{0,m,n}(t,s), \q (t,s)\in\D_*[0,T],\\
&{d \cY_{m,n}(t,s)\over ds}=0,\q (t,s)\in\D^*[0,T],\\
&\cX_{m,n}(t,0)=x_{m,n}(t),\q \cY_{m,n}(t,T)=\cX_{m,n}(t,t),\q t\in[0,T].\ea\right.$$
Noting that $\cY_{m,n}(t,t)=\cX_{m,n}(t,t)$ for any $t\in[0,T]$, we get
\begin{align*}
&{d\over dt}\[\int_t^T \lan \cX_{m,n}(s,t),\cY_{m,n} (s,s)\ran ds+\lan \h G(\cX_{m,n}(T,T)),\cX_{m,n}(T,t)\ran\]\\
%
%
&\q=- \lan \cX_{m,n}(t,t),\cX_{m,n} (t,t)\ran+\int_t^T\Big \lan \cY_{m,n} (s,s) , \[f_{0,m,n}(s,t)- \int_t^T \cY_{m,n}(r,r)dr\\
&\qq-\h G(\cX_{m,n}(T,T))\]\Big\ran ds+\Big\lan\h G(\cX_{m,n}(T,T)),\[f_{0,m,n}(T,t)- \int_t^T \cY_{m,n}(r,r)dr\\
&\qq-\h G(\cX_{m,n}(T,T))\]\Big\ran\\
&\q=- \lan \cX_{m,n}(t,t),\cX_{m,n} (t,t)\ran-\Big\lan\h G(\cX_{m,n}(T,T))+\int_t^T \cY_{m,n} (s,s)ds,\,\\
&\qq\q \h G(\cX_{m,n}(T,T))+\int_t^T \cY_{m,n} (s,s)ds\Big\ran+\int_t^T\lan \cX_{m,n}(s,s),\,f_{0,m,n}(s,t)\ran ds\\
&\qq+\lan\h G(\cX_{m,n}(T,T)),f_{0,m,n}(T,t)\ran,\q t\in[0,T].
\end{align*}
Further, noting $\cX_{m,n}(t,0)=x_{m,n}(t)$ for any $t\in[0,T]$, we have
\begin{align*}
&\lan\h G(\cX_{m,n}(T,T)),\cX_{m,n}(T,T)\ran-\int_0^T \lan x_{m,n}(s),\cX_{m,n} (s,s)\ran ds-\lan \h G(\cX_{m,n}(T,T)), x_{m,n}(T)\ran\\
&\q=-\int_0^T\[ \big|\cX_{m,n}(t,t)\big|^2+\Big| \h G(\cX_{m,n}(T,T))+\int_t^T \cY_{m,n} (s,s)ds\Big|^2\]dt\\
&\qq\, +\int_0^T\int_t^T\lan \cX_{m,n}(s,s),\,f_{0,m,n}(s,t)\ran dsdt +\int_0^T\lan \h G(\cX_{m,n}(T,T)),f_{0,m,n}(T,t)\ran dt.
\end{align*}
By Young's inequality, from \rf{MC-3}, the above implies that
\begin{align*}
&| G_0^{1\over 2}\cX_{m,n}(T,T)|^2 +\int_0^T\[ \big|\cX_{m,n}(t,t)\big|^2+\Big|  \h G(\cX_{m,n}(T,T))+\int_t^T \cY_{m,n} (s,s)ds\Big|^2\]dt\\
&\q\les \e \int_0^T \big|\cX_{m,n}(t,t)\big|^2dt +\e| G_0^{1\over 2}\cX_{m,n}(T,T)|^2
+C_\e \int_0^T \big|x_{m,n}(t)\big|^2dt\\
&\qq +C_\e| x_{m,n}(T)|^2+C_\e \int_0^T\int_t^T \big|f_{0,m,n}(s,t)\big|^2dsdt+C_\e\int_0^T | f_{0,m,n}(T,t)|^2dt,
\end{align*}
which yields that
\begin{align*}
&| G_0^{1\over 2}\cX_{m,n}(T,T)|^2 +\int_0^T\[ \big|\cX_{m,n}(t,t)\big|^2+\Big|  \h G(\cX_{m,n}(T,T))+\int_t^T \cY_{m,n} (s,s)ds\Big|^2\]dt\\
&\q\les C \[\int_0^T \big|x_{m,n}(t)\big|^2dt+| x_{m,n}(T)|^2+\int_0^T\int_t^T \big|f_{0,m,n}(s,t)\big|^2dsdt\\
&\qq\qq+\int_0^T | f_{0,m,n}(T,t)|^2dt\].
\end{align*}
Then noting that
$$| \h G(\cX_{m,n}(T,T))|\les K |G_0^{1\over 2}\cX_{m,n}(T,T)|,$$
by the standard results of VIEs, we have
\begin{align*}
\sup_{t,s}|\cX_{m,n}(t,s)|^2&\les  C \[\int_0^T \big|x_{m,n}(t)\big|^2dt+| x_{m,n}(T)|^2+\sup_t\int_0^t \big|f_{0,m,n}(t,s)\big|^2ds\\
&\qq+\int_0^T | f_{0,m,n}(T,t)|^2dt+\sup_t|x_{m,n}(t)|^2\].\end{align*}
Thus, $\{\cX_n(\cd,\cd)\}_{n>0}$ is a Cauchy sequence in the space of continuous functions.
With the fact that $\cY_{m,n}(t,s)=\cX_{m,n}(t,t)$, we know that $\{\cY_n(\cd,\cd)\}_{n>0}$
is also a Cauchy sequence in the space of continuous functions.
Let $(\cX(\cd,\cd),\cY(\cd,\cd))$ be the limit of $\{\cX_n(\cd,\cd),\cY_n(\cd,\cd)\}_{n>0}$ as $n\to\infty$.
Taking $n\to\infty$ in \rf{X-m,n}, we know that  $(\cX(\cd,\cd),\cY(\cd,\cd))$ is a solution to \rf{linear}.
\end{proof}

From the proof, we can see that \rf{linear} is essentially an FBDE.
Next, we consider the following family of FBVIEs parameterized by $\a \in[0,1]$:
\bel{para}\left\{\ba
\cX_s(t,s)&=\a f\Big(t,s,\cX(s,s),\cY(s,s), \int_s^T K(s,r)\cY(r,r)dr,\cX(T,T)\Big)\\
&\q+(1-\a)\[- \int_s^T \cY(r,r)dr-G(\cX(T,T))\]+f_0(t,s),\\
& \q\qq (t,s)\in\D_*[0,T],\\
\cY_s(t,s)&=-\big[ \a g(t,s,\cX(s,s),\cY(s,s))+g_0(t,s) \big],\q (t,s)\in\D^*[0,T],\\
\cX(t,0)&=x(t),\q t\in[0,T],\\
 \cY(t,T)&=\a h(t,\cX(t,t),\cX(T,T))+(1-\a)\cX(t,t)+h_0(t),\q t\in[0,T],\ea\right.\ee
where $f_0(\cd,\cd)$, $g_0(\cd,\cd)$, and $h_0(\cd)$ are continuous functions.
Clearly, when $\a=1$, the existence of the solution of \rf{para} implies that of \rf{FBVIE-main1}.

\begin{lemma}\label{lem:con}
Let {\rm \ref{ass:A1}} and {\rm \ref{ass:A2}} hold.
We assume that, for a given $\a_0\in[0,1)$ and for any $f_0(\cd,\cd),g_0(\cd,\cd),h_0(\cd)$,
\rf{para} admits a unique solution. Then there exists a constant $\d_0\in(0,1)$, such that for all
$\a \in[\a_0,\a_0+\d_0]$, and for any $f_0(\cd,\cd),g_0(\cd, \cd), h_0(\cd)$,
\rf{para} admits a unique solution.
\end{lemma}

\begin{proof}
Note that for each $f_0(\cd,\cd)$, $g_0(\cd,\cd)$, $h_0(\cd)$, and $\a_0\in[0,1)$,
\rf{para} admits a unique solution. Then, for each pair $(\f(\cd,\cd),\p(\cd,\cd))$,
there exists a unique  pair $(\cX(\cd,\cd),\cY(\cd,\cd))$ satisfying the following FBVIE:
$$
\left\{\ba
\cX_s(t,s)&=\a_0 f\Big(t,s,\cX(s,s),\cY(s,s), \int_s^T K(s,r)\cY(r,r)dr,\cX(T,T)\Big)\\
&\q+(1-\a_0)\[- \int_s^T \cY(r,r)dr-G(\cX(T,T))\]\\
&\q+\d f\Big(t,s,\f(s,s),\p(s,s), \int_s^T K(s,r)\p(r,r)dr,\f(T,T)\Big)\\
&\q+\d \[\int_s^T \phi(r,r)dr+G(\f(T,T))\]+f_0(t,s),\q (t,s)\in\D_*[0,T], \\
\cY_s(t,s)&=-\a_0  g(t,s,\cX(s,s),\cY(s,s))-\d  g(t,s,\f(s,s),\p(s,s))-g_0(t,s),\\
&\qq\qq\qq (t,s)\in\D^*[0,T],\\
\cX(t,0)&=x(t),\q t\in[0,T],\\
\cY(t,T)&=\a_0 h(t,\cX(t,t),\cX(T,T))+(1-\a_0)\cX(t,t)\\
&\q+\d\[h(t,\f(t,t),\f(T,T))-\f(t,t)\]+h_0(t),\q t\in[0,T].\ea\right.
$$

\ms

\ms

We are going to prove that the mapping $\G_{\a_0+\d}[\cd,\cd]$, defined by
\begin{align*}
&\G_{\a_0+\d}[\f(\cd,\cd),\p(\cd,\cd)]=(\cX(\cd,\cd),\cY(\cd,\cd)),\\
&\qq \forall (\f(\cd,\cd),\p(\cd,\cd))\in C(\D_*[0, T] ;\dbR^n) \times C(\D^*[0, T] ;\dbR^n),
\end{align*}
is a contraction.
Let $(\f_i(\cd,\cd),\p_i(\cd,\cd))\in C(\D_*[0, T] ;\dbR^n) \times C(\D^*[0, T] ;\dbR^n)$, and
$$
(\cX_i(\cd,\cd),\cY_i(\cd,\cd))=\G_{\a_0+\d}[\f_i(\cd,\cd),\p_i(\cd,\cd)],\q \hbox{for } i=1,2.
$$
Denote $(\h\f(\cd,\cd), \h\p(\cd,\cd))=(\f_1(\cd,\cd)-\f_2(\cd,\cd),\p_1(\cd,\cd)-\p_2(\cd,\cd))$ and
$(\h\cX(\cd,\cd), \h\cY(\cd,\cd))=(\cX_1(\cd,\cd)-\cX_2(\cd,\cd),\cY_1(\cd,\cd)-\cY_2(\cd,\cd))$.
Similar to \rf{MC-2}, we can denote
\begin{align*}
&\h h(t;\f),\q \h f(t,s;\f,\p),\q\h g(t,s;\f,\p), \q\hbox{and}\\
& \h h(t;\cX),\q \h f(t,s;\cX,\cY),\q\h g(t,s;\cX,\cY),
\end{align*}
with $(x_i(\cd),y_i(\cd))$ replaced by $(\f_i(\cd,\cd),\p_i(\cd,\cd))$
and $(\cX_i(\cd,\cd),\cY_i(\cd,\cd))$, respectively. Next, we denote
$\h G(\h\cX(T,T))=G(\cX_1(T,T))-G(\cX_2(T,T))$ and $\h G(\h\f(T,T))=G(\f_1(T,T))-G(\f_2(T,T))$. Then,
$$
\left\{\ba
{d \h\cX(t,s)\over ds}&=\a_0 \h f(t,s;\cX,\cY)+(1-\a_0)\[- \int_s^T \h\cY(r,r)dr-\h G(\h\cX(T,T))\]\\
&\q+\d \h f(t,s;\f,\p)+\d \[\int_s^T \h\phi(r,r)dr+\h G(\h\f(T,T))\], \q (t,s)\in\D_*[0,T],\\
{d\h\cY(t,s)\over ds}&=-\a_0 \h g(t,s;\cX,\cY)-\d \h g(t,s;\f,\p),\q (t,s)\in\D^*[0,T],\\
\h\cX(t,0)&=0,\q t\in[0,T],\\
\h\cY(t,T)&=\a_0\h h(t;\cX)+(1-\a_0)\h\cX(t,t)+\d[\h h(t;\f)-\h\f(t,t)],\q t\in[0,T].\ea\right.
$$
By the same argument as that employed in the proof of \autoref{prop:uniqueness}, we have
\begin{align*}
&{d\over dt}\[\int_t^T \lan \h\cY(s,s),\h\cX(s,t)\ran ds+\lan\h G(\h\cX(T,T)),\h\cX(T,t)\ran\]\\
%
%
&\q=-\Big\lan  \a_0\h h(t;\cX)+(1-\a_0)\h\cX(t,t)+\d[\h h(t;\f)-\h\f(t,t)]\\
&\qq\qq\qq+\int_t^T [\a_0 \h g(t,s;\cX,\cY)+\d \h g(t,s;\f,\p)]ds,\, \h\cX(t,t)\Big\ran\\
&\qq+\int_t^T\Big\lan \h\cY(s,s),\,\a_0 \h f(s,t;\cX,\cY)+(1-\a_0)\[- \int_t^T \h\cY(r,r)dr-\h G(\h\cX(T,T))\]\\
&\qq\qq\qq +\d \h f(s,t;\f,\p)+\d \[\int_t^T \h\phi(r,r)dr+\h G(\h\f(T,T))\]\Big\ran ds\\
&\qq+\Big\lan \h G(\h\cX(T,T)),\,\a_0 \h f(T,t;\cX,\cY)+(1-\a_0)\[- \int_t^T \h\cY(r,r)dr-\h G(\h\cX(T,T))\]\\
&\qq\qq\qq +\d \h f(T,t;\f,\p)+\d \[\int_t^T \h\phi(r,r)dr+\h G(\h\f(T,T))\]\Big\ran,\q t\in[0,T].
\end{align*}
Then from the non-local monotonicity condition \rf{MC-1}, we have
\begin{align*}
&{d\over dt}\[\int_t^T \lan \h\cY(s,s),\h\cX(s,t)\ran ds+\lan\h G(\h\cX(T,T)),\h\cX(T,t)\ran\]\\
&\q\les-\d\Big\lan  \h h(t;\f)-\h\f(t,t)+\int_t^T  \h g(t,s;\f,\p)ds,\, \h\cX(t,t)\Big\ran\\
&\qq+\d \int_t^T\Big\lan \h\cY(s,s),\, \h f(s,t;\f,\p)+\int_t^T \h\phi(r,r)dr+\h G(\h\f(T,T))\Big\ran ds\\
&\qq+\d\Big\lan\h G(\h\cX(T,T)),\,\h f(T,t;\f,\p)+\int_t^T \h\phi(r,r)dr+\h G(\h\f(T,T))\Big\ran\\
&\qq -\a_0 \g|\h\cX(t,t)|^2-(1-\a_0)\[|\h\cX(t,t)|^2+\Big|\int_t^T \h\cY(r,r)dr+\h G(\h\cX(T,T))\Big|^2\],\q t\in[0,T].
\end{align*}
Thus, noting $\h\cX(\cd,0)\equiv 0$ and $\g>0$, by \rf{MC-3}, we get
\begin{align*}
&\big| G_0^{1\over 2}\h\cX(T,T)\big|^2+\int_0^T|\h\cX(t,t)|^2dt\\
&\q\les \d C\int_0^T \bigg\{\Big|\Big\lan  \h h(t;\f)-\h\f(t,t)+\int_t^T  \h g(t,s;\f,\p)ds,\, \h\cX(t,t)\Big\ran\Big|\\
&\qq\qq+\Big|\int_t^T\Big\lan \h\cY(s,s),\, \h f(s,t;\f,\p)+\int_t^T \h\phi(r,r)dr+\h G(\h\f(T,T))\Big\ran ds\Big|\\
&\qq\qq+\Big|\Big\lan \h G(\h \cX(T,T)),\,\h f(T,t;\f,\p)+\int_t^T \h\phi(r,r)dr+\h G(\h\f(T,T))\Big\ran\Big|\bigg\}dt\\
&\q\les {1\over 2} |G_0^{1\over2}\h\cX(T,T)|^2+{1\over 2}\int_0^T|\h\cX(t,t)|^2dt+\d C\int_0^T|\h\cY(t,t)|^2dt\\
&\qq +\d C|\h\f(T,T)|^2+ \d C\int_0^T \bigg\{|  \h h(t;\f)|^2+|\h\phi(t,t)|^2+|\h f(T,t;\f,\p)|^2+|\h\f(t,t)|^2\\
&\qq\qq+\int_t^T\big[ | \h g(t,s;\f,\p)|^2 +|\h f(s,t;\f,\p)|^2\big]ds\bigg\}dt,
\end{align*}
where the last inequality is obtained by applying the Young's inequality.
Thus,
\begin{align*}
&\big| G_0^{1\over 2}\h\cX(T,T)\big|^2+\int_0^T|\h\cX(t,t)|^2dt\\
&\q\les \d C\int_0^T|\h\cY(t,t)|^2dt +\d C|\h\f(T,T)|^2+ \d C\int_0^T \big[|\h\f(t,t)|^2+|\h\p(t,t)|^2\big]dt.
\end{align*}
By the standard results of BVIEs, we have
\bel{Prop:con-p1}
\begin{aligned}
&\int_0^T |\h\cY(t,t)|^2dt \les\d  C\[\int_0^T\(|\h\f(t,t)|^2+|\h\p(t,t)|^2\)dt+|\h\f(T,T)|^2\]\\
&\qq +C\[\int_0^T|\h\cX(t,t)|^2dt+|G_0\h\cX(T,T)|^2\]\\
&\q\les  \d C\[\int_0^T\(|\h\f(t,t)|^2+|\h\p(t,t)|^2\)dt+|\h\f(T,T)|^2\]\\
&\qq +C\[\int_0^T|\h\cX(t,t)|^2dt+|G_0^{1\over 2}\h\cX(T,T)|^2\].
\end{aligned}
\ee
Combining the above two estimates together, we have
\begin{align*}
& |G_0^\frac{1}{2}\h\cX(T,T)|^2+\int_0^T|\h\cX(t,t)|^2dt \\
&\q\les \d C\[\int_0^T\(|\h\f(t,t)|^2+|\h\p(t,t)|^2\)dt+\int_0^T|\h\cX(t,t)|^2dt+|G_0^{1\over 2}\h\cX(T,T)|^2+|\h\f(T,T)|^2\],
\end{align*}
which implies that
\begin{align*}
& |G_0^\frac{1}{2}\h\cX(T,T)|^2+\int_0^T|\h\cX(t,t)|^2dt \\
&\q\les \d C\[\int_0^T\(|\h\f(t,t)|^2+|\h\p(t,t)|^2\)dt+|\h\f(T,T)|^2\].
\end{align*}
Substituting the above into \rf{Prop:con-p1}, we get
\begin{align*}
&|G_0^\frac{1}{2}\h\cX(T,T)|^2+\int_0^T\(|\h\cX(t,t)|^2+|\h\cY(t,t)|^2\)dt \\
&\q\les \d C\[\int_0^T\(|\h\f(t,t)|^2+|\h\p(t,t)|^2\)dt+|\h\f(T,T)|^2\].
\end{align*}
Then by the standard estimates of FVIEs, we have
\bel{Prop:con-p2}
\begin{aligned}
&\sup _{t,s} |\h\cX(t,s)|^2 \les \d C\[\int_0^T\(|\h\f(t,t)|^2+|\h\p(t,t)|^2\)dt+|\h\f(T,T)|^2\]\\
&\qq+C\[ |G_0^\frac{1}{2}\h\cX(T,T)|^2+\int_0^T|\h\cY(t,t)|^2dt\]\\
&\q\les \d C\[\int_0^T\(|\h\f(t,t)|^2+|\h\p(t,t)|^2\)dt+|\h\f(T,T)|^2\].
\end{aligned}
\ee
On the other hand, the standard estimate of BVIEs implies that
\bel{Prop:con-p3}
\begin{aligned}
&\sup _{t,s} |\h\cY(t,s)|^2 \les \d C\[\int_0^T|\h\p(t,t)|^2dt+\sup _{t} |\h\f(t,t)|^2\]+C\sup _{t} |\h\cX(t,t)|^2\\
&\q\les \d C\[\int_0^T|\h\p(t,t)|^2dt+\sup _{t} |\h\f(t,t)|^2\],
\end{aligned}
\ee
in which the last equality is due to \rf{Prop:con-p2}.
Combining the estimates \rf{Prop:con-p2} and \rf{Prop:con-p3} together, we get
\begin{align*}
\sup_{t,s}\[|\h\cX(t,s)|^2+|\h\cY(t,s)|^2\] \les \d C\sup_{t,s}\[|\h\f(t,s)|^2+|\h\p(t,s)|^2\].
\end{align*}
We now choose $\d_0={1\over 2C}$, which is independent of $\a_0$.
Clearly, for each fixed $\d\in [0,\d_0]$, the mapping $\G_{\a_0+\d}[\cd,\cd]$ is a contraction.
It turns out that this mapping has a unique fixed point $(\cX^{\a_0+\d}(\cd,\cd), \cY^{\a_0+\d}(\cd,\cd))$,
which is the unique solution of \rf{para} for $\a=\a_0+\d$.
\end{proof}

\ms

We  are ready to give the proof of \autoref{Thm:well-posedness} now.

\begin{proof}
From \autoref{lemm:alpha=0}, we see immediately that, when $\a=0$,
for any $f_0(\cd,\cd)$, $g_0(\cd,\cd)$, and $h_0(\cd)$, FBVIE \rf{para} admits a unique solution.
Then by  \autoref{lem:con}, for any $f_0(\cd,\cd)$, $g_0(\cd,\cd)$, and $h_0(\cd)$,
we can solve the FBVIE \rf{para} successively for the case $\a \in[0, \d_0],[\d_0, 2 \d_0], \ldots, [(N-1)\d_0, 1]$, with $(N-1)\d_0<1\les N\d_0$.
 It turns out that, when $\a=1$, for any  $f_0(\cd,\cd)$, $g_0(\cd,\cd)$, and $h_0(\cd)$,  \rf{para} admits a solution,
which implies that the solution of \rf{FBVIE-main1} exists.
The uniqueness of solutions to \rf{FBVIE-main1} is an immediate consequence of \autoref{prop:uniqueness}.
\end{proof}

\section{Examples} \label{sec:application}

In this section, we shall give two explicit examples of  coupled FBVIEs,
whose solvability can be obtained from \autoref{Thm:well-posedness}.

\begin{example}[Hamiltonian System Derived From Linear-Convex Optimal Control Problems]
Consider the controlled  Volterra integral equation:
$$
X(t)=x(t)+\int_0^t\[A(t, s) X(s)+B(t, s) u(s)\]ds,\q t \in[0,T],
$$
and the cost functional:
$$
J( u(\cd))=\int_0^T\[Q(t,X(t))+\lan R(t) u(t),u(t)\ran\] d t+M(X(T)).
$$
We assume that all the functions involved  above are smooth, and
\begin{align*}
&|Q(t,x)|\les L(1+|x|^2),\q|M(x)|\les L(1+|x|^2),\q \forall x\in\dbR^n, \\
&  R(t)\ges \d I_m,\q \lan Q_{x}(t,x_1)-Q_{x}(t,x_2),\,x_1-x_2\ran\ges \d |x_1-x_2|^2,\\
&|Q_{x}(t,x_1)-Q_{x}(t,x_2)|\les L|x_1-x_2|,\\
& \lan M_{x}(x_1)-M_{x}(x_2),\,x_1-x_2\ran\ges  |G_0^{1\over 2} (x_1-x_2)|^2,\\
&| M_{x}(x_1)-M_{x}(x_2)|\les L|G_0^{1\over 2} (x_1-x_2)|,\q \forall x_1,x_2\in\dbR^n,
\end{align*}
for some $\d>0$, $L>0$, and $G_0\ges 0$.
A typical example   is that
\bel{LQ-Form}
Q(t,x)=\lan Q(t)x,x\ran, \q \hbox{with } Q(t)\ges \d I_n>0, \q M(x)=\lan G_0 x,x\ran,
\ee
under which the problem has a linear-quadratic form.
In addition, if the state $X(\cd)$ is one-dimensional,  we can also take
$$
Q(t,x)={x^4-2\over x^2+1},
$$
by which the problem is out of the linear-quadratic framework.
The optimal control problem can be stated as follows:
Find a control $\bar u(\cd)\in L^2([0,T];\dbR^m)$ such that
$$
J(0, x(\cd);\bar u(\cd))=\inf_{u(\cd)\in L^2([0,T];\dbR^m)} J(0, x(\cd);u(\cd)).
$$

\ms

Clearly, this is the optimal control problem with the form \rf{LCONVEX} given in Introduction.
Then from \rf{LCONVEX1} and \rf{OS-1}, the optimal control $\bar u(\cd)$ can be written as:
\begin{align}\label{MP-condition}
\bar{u}(s)=-{1\over 2}R(s)^{-1}\int_s^T B(r,s)^\top Y(r)dr-{1\over 2}R(s)^{-1}B(T,s)^\top M_x( X(T)), \quad s \in[0, T],
\end{align}
with the Hamiltonian system
\bel{MP-system}\left\{\begin{aligned}
X(t)= & x(t)+\int_0^t \[A(t,s) X(s)-{1\over 2}B(t,s) R(s)^{-1}\int_s^T B(r,s)^\top Y(r)dr\\
&\qq-{1\over 2}B(t,s)R(s)^{-1}B(T,s)^\top M_x(X(T))\] ds, \\
Y(t)= & Q_x(t,X(t))+A(T,t)^\top M_x(X(T))+\int_t^T A(s,t)^{\top} Y(s)ds,\end{aligned}\right.
\q t\in[0,T].
\ee
Then,
\begin{align*}
&f(t,s,x,y,y^\prime,x^\prime)=A(t,s)x-{1\over 2}B(t,s) R(s)^{-1}y^\prime-{1\over 2}B(t,s)R(s)^{-1}B(T,s)^\top M_x(x^\prime),\\
& K(s,r)=B(r,s)^\top,\q g(t,s,x,y)=A(s,t)^\top y,\q h(t,x,x^\prime)=Q_x(t,x)+A(T,t)^\top M_x(x^\prime).
\end{align*}
For any $x_i(\cd),y_i(\cd)\in C([0,T];\dbR^n),\,i=1,2$,
denote $\h x(\cd)$, $\h y(\cd)$, $\h f(\cd,\cd)$,  $\h g(\cd,\cd)$, and   $\h h(\cd)$ as that in \rf{MC-2}.
Moreover, we take $G(\cd)=M_x(\cd)$, and denote
$$
\h G(T)=G (x_1(T))-G (x_2(T)),\q \h Q_x(t)=Q_x(t,x_1(t))-Q_x(t,x_2(t)).
$$
Clearly, \rf{MC-3} holds.
Then,
\begin{align}
&\int_t^T \big\lan \h y(s),\, \h f(s,t)\big\ran ds-\Big\lan \h h(t)+\int_t^T\h g(t,s)ds,\, \h x(t) \Big\ran+\big\lan\h G (T), \,\h f(T,t)\big\ran \nn\\
&\q=\int_t^T \bigg\lan \h y(s),\, A(s,t)\h x(t)-{1\over 2}B(s,t) R(t)^{-1}\int_t^TB(r,t)^\top \h y(r)dr\nn\\
&\qq\qq-{1\over 2}B(s,t)R(t)^{-1}B(T,t)^\top \h G (T)\bigg\ran ds\nn\\
&\qq-\Big\lan\h Q_x(t) +A(T,t)^\top \h G (T)+\int_t^T A(s,t)^\top\h y(s) ds,\, \h x(t) \Big\ran\nn\\
&\qq+\bigg\lan \h G (T), \,A(T,t)\h x(t)-{1\over 2}B(T,t) R(t)^{-1}\int_t^TB(r,t)^\top \h y(r)dr\nn\\
&\qq\qq -{1\over 2} B(T,t)R(t)^{-1} B(T,t)^\top\h G (T)\bigg\ran\nn\\
&\q =-\big\lan\h Q_x(t),\, \h x(t)\big\ran-{1\over 2}\Big\lan R(t)^{-1}\[B(T,t)^\top \h G (T)\nn\\
&\qq\,+\int_t^TB(r,t)^\top \h y(r)dr\],\, \[B(T,t)^\top \h G (T)+\int_t^TB(r,t)^\top \h y(r)dr\]\Big\ran\nn\\
&\q\les -\d |\h x(t)|^2,\q t\in[0,T],\nn
\end{align}
which implies the non-local monotonicity condition \rf{MC-1} holds.
Thus, from \autoref{Thm:well-posedness}, FBVIE \rf{MP-system} admits a unique solution.
It turns out that the control process $\bar u(\cd)$, given by \rf{MP-condition}, is the unique optimal control.
Moreover, if $Q(\cd,\cd)$ and $M(\cd)$ admit the quadratic form \rf{LQ-Form}, then the value function can be given by
$$
\begin{aligned}
&\int_0^T \lan \cX(s,0),\cY(s,s)\ran ds+\lan G_0\cX(T,0),\cX(T,T)\ran\\
&\q = \int_0^T \big[\lan Q(s)X(s),X(s)\ran +\lan R(s)\bar u(s),\bar u(s)\ran \big]ds+\lan G_0X(T),X(T)\ran,
\end{aligned}
$$
where $(\cX(\cd,\cd),\cY(\cd,\cd))$ satisfies the following auxiliary system corresponding to \rf{MP-system}:
$$
\left\{\begin{aligned}
&\cX_s(t,s)=  A(t,s)\cX(s,s)-{1\over 2}B(t,s)R(s)^{-1}\int_s^T B(r,s)^\top\cY(r,r)dr\\
&\qq- B(t,s)R(s)^{-1}B(T,s)^\top G_0\cX(T,T),\q (t,s)\in\D_*[0,T], \\
&\cY_s(t,s)=  - A(s,t)^{\top}\cY(s,s),\q (t,s)\in\D^*[0,T], \\
& \cX(t,0)=x(t),\q \cY(t,T)=2 Q(t)\cX(t,t)+2A(T,t)^\top G_0\cX(T,T),\q t\in[0,T].
\end{aligned}\right.
$$
\end{example}

The following is an FBVIE with some general nonlinear terms. We will show that under some proper assumptions,
it also satisfies the non-local monotonicity condition \rf{MC-1}--\rf{MC-2}.

\begin{example}[Nonlinear FBVIEs]
Consider the following nonlinear FBVIE:
\bel{FBVIE-Nonlinear}\left\{\begin{aligned}
X(t)= & x(t)+\int_0^t \[A(t,s) X(s)+B(t,s)a(s,X(s))\\
&\qq-B(t,s) \int_s^T B(r,s)^\top Y(r)dr\] ds, \\
Y(t)= & b(t,X(t))+\phi\Big(t,\int_t^T B(r,t)^\top Y(r)dr\Big)\\
&\qq+\int_t^T\[ A(s,t)^{\top} Y(s)+\psi(t,s,X(t))\] ds,\end{aligned}\right.\q t\in[0,T].
\ee
Assume that  all the coefficients of the above are continuous functions. Moreover, there exist constants $\l,L_a,L_b,L_\phi,L_\psi>0$ such that
\begin{align*}
&|a(s,x_1)-a(s,x_2)|\les L_a |x_1-x_2|,\q |b(s,x_1)-b(s,x_2)|\les L_b |x_1-x_2|,\\
& \lan b(s,x_1)-b(s,x_2) ,x_1-x_2\ran\ges\l |x_1-x_2|^2,\\
&|\psi(t,s,x_1)-\psi(t,s,x_2)|\les L_\psi |x_1-x_2|,\q \forall x_1,x_2\in\dbR^n;\\
&|\phi(s,y^\prime_1)-\phi(s,y^\prime_2)|\les L_\phi |y^\prime_1-y^\prime_2|,\q \forall y^\prime_1,y^\prime_2\in\dbR^n.
\end{align*}
We now show that under proper conditions, the above satisfies the monotonicity condition \rf{MC-1}.
Indeed,
\begin{align*}
&\int_t^T \big\lan \h y(s),\, \h f(s,t)\big\ran ds-\Big\lan \h h(t)+\int_t^T\h g(t,s)ds,\, \h x(t) \Big\ran\\
&\q =\int_t^T \Big\lan \h y(s),\, \[A(s,t)\h x(t)+B(s,t)[a(t,x_1(t))-a(t,x_2(t))]\\
&\qq -B(s,t) \int_t^T B(r,t)^\top \h y (r)dr\]\Big\ran ds-\Big\lan\h x(t),\, \[b(t,x_1(t))-b(t,x_2(t))\\
&\qq+\phi\Big(t,\int_t^T B(r,t)^\top y_1(r)dr\Big)-\phi\Big(t,\int_t^T B(r,t)^\top y_2(r)dr\Big)\\
&\qq +\int_t^T\( A(s,t)^{\top} \h y(s)+\psi(t,s,x_1(t))-\psi(t,s,x_2(t))\) ds\]\Big\ran\\
&\q = \Big\lan\int_t^T B(s,t)^\top\h y(s)ds,\, \big[a(t,x_1(t))-a(t,x_2(t))\big]\Big\ran\\
&\qq -\Big| \int_t^T B(r,t)^\top \h y (r)dr\Big|^2-\Big\lan\h x(t),\, \[b(t,x_1(t))-b(t,x_2(t))\\
&\qq+\phi\Big(t,\int_t^T B(r,t)^\top y_1(r)dr\Big)-\phi\Big(t,\int_t^T B(r,t)^\top y_2(r)dr\Big)\]\Big\ran\\
&\qq -\Big\lan\h x(t),\,\int_t^T\big[\psi(t,s,x_1(t))-\psi(t,s,x_2(t))\big] ds\Big\ran\\
&\q \les {1\over 2} \Big| \int_t^T B(r,t)^\top \h y (r)dr\Big|^2+ {1\over 2}L_a^2 |\h x(t)|^2-\Big| \int_t^T B(r,t)^\top \h y (r)dr\Big|^2\\
&\qq -\l|\h x(t)|^2+{1\over 2} \Big| \int_t^T B(r,t)^\top \h y (r)dr\Big|^2+ {1\over 2}L_\phi^2 |\h x(t)|^2+L_\psi T|\h x(t)|^2\\
&\q=-\Big(\l-{1\over 2}L_a^2 -{1\over 2}L_\phi^2-L_\psi T\Big) |\h x(t)|^2,\q t\in[0,T].
\end{align*}
Let
$$
\l-{1\over 2}L_a^2 -{1\over 2}L_\phi^2-L_\psi T>0,
$$
and take $G(\cd)\equiv0$.
Then the non-local monotonicity condition \rf{MC-1}--\rf{MC-2} holds.
One can  see that for the case with only one nonlinear term $b(\cd)$, that is $a(\cd),\phi(\cd),\psi(\cd)=0$,
we just need to assume that $\l>0$. Thus, by \autoref{Thm:well-posedness}, \rf{FBVIE-Nonlinear} admits a unique solution.
\end{example}

\section{Conclusion}

The main contribution of this paper is that we provide a method for finding a non-local monotonicity condition for
coupled forward-backward Volterra integral equations, under which the system admits a unique solution.
From this procedure, we can  see the celebrated method of continuation developed for establishing
the solvability of FBSDEs (see \cite{Hu-Peng1995,Yong1997,Peng-Wu1999})
is also an optimal control approach in some sense.
The stochastic version of coupled FBVIEs  has more potential in applications; for example,
it can be applied in the popular rough Heston model, and it can provide a probabilistic interpretation
for non-local PDEs. A natural problem is to extend the current result to
stochastic FBVIEs. Although  the main methods introduced in the current paper still work,
it becomes  more technical. In the current paper, we mainly focus on  the basic idea of finding
a non-local monotonicity condition, and do not hope to touch too many technical issues.
We will report the related results of coupled stochastic FBVIEs separately  in the near future.


\begin{thebibliography}{99}

\bibitem{AIOmari-Gourley2003} J. F. M. Al-Omari and S. A. Gourley,
\it Stability and traveling fronts in Lotka-Volterra competition models with stage structure,
\rm SIAM J. Appl. Math., {\bf 63} (2003), 2063--2086.

\bibitem{Angell1976}
T. S. Angell, On the optimal control of systems governed by nonlinear Volterra equations, J. Optim.
Theory Appl., {\bf19} (1976), 29--45.

\bibitem{Bondi2023}
A. Bondi and F. Flandoli,
\it On the Kolmogorov equation associated with Volterra equations and Fractional Brownian Motion,
\rm  arxiv:2309.13597, 2023.

\bibitem{Bonesini2023}
O. Bonesini, A. Jacquier, and A. Pannier,
\it  Rough volatility, path-dependent PDEs and weak rates of convergence,
\rm arxiv:2304.03042, 2023.



\bibitem{Carlson1987} D. A. Carlson,
\it An elementary proof of the maximum principle for optimal control problems governed by a Volterra integral equation,
\rm J. Optim. Theory Appl., {\bf54} (1987), 43--61.

\bibitem{Comte1998}
F. Comte and E. Renault,
\it Long memory in continuous-time stochastic volatility models,
\rm Math. Finance, {\bf8} (1998), 291--323.

\bibitem{Cvitanic1996} J. Cvitani\'{c} and J. Ma,
\it  Hedging options for a large investor and forward-backward SDE's,
\rm Ann. Appl. Probab., {\bf 6} (1996), 370--398.

\bibitem{Delarue2002} F. Delarue,
\it On the existence and uniqueness of solutions to FBSDEs in a non-degenerate case,
\rm Stochastic Process. Appl., {\bf 99} (2002),  209--286.

\bibitem{ElEuch2018}
O. El Euch and M. Rosenbaum,
\it Perfect hedging in rough Heston models,
\rm Ann. Appl. Probab., {\bf28} (2018), 3813--3856.

\bibitem{ElEuch2019}
 O. El Euch and M. Rosenbaum,
\it The characteristic function of rough Heston models,
\rm Math. Finance, {\bf 29} (2019), 3--38.

\bibitem{Gopalsamy1980} K. Gopalsamy,
\it Time lags and global stability in two-species competition,
\rm Bull. Math. Biol., {\bf 42} (1980), 729--737.


\bibitem{Hamaguchi2021} Y. Hamaguchi,
\it Small-time solvability of a flow of forward-backward stochastic differential equations,
\rm Appl. Math. Optim., {\bf 84} (2021), 567--588.

\bibitem{Hamaguchi2021-1} Y.~Hamaguchi,
\it Extended backward stochastic Volterra integral equations and their applications to time-inconsistent stochastic recursive control problems,
\rm Math. Control Relat. Fields, {\bf11} (2021), 433--478.


\bibitem{Hu-Jin-Zhou2012} Y.~Hu, H.~Jin, and X.~Y.~Zhou,
\it Time-inconsistent stochastic linear--quadratic control,
\rm SIAM J. Control Optim., {\bf 50} (2012), 1548--1572.

\bibitem{Hu-Peng1995} Y.~Hu and S.~Peng,
\it Solution of forward-backward stochastic differential equations,
\rm Probab. Theory Related Fields, {\bf103} (1995), 273--283.

\bibitem{Kot2001} M. Kot,
\it Elements of mathematical ecology,
\rm Cambridge University Press, Cambridge, 2001.

\bibitem{Lin2020}
P. Lin and J. Yong,
\it Controlled singular Volterra integral equations and Pontryagin maximum principle,
\rm SIAM J. Control Optim., {\bf 58} (2020),  136--164.

\bibitem{Ma-Protter-Yong1994} J.~Ma, P. Protter, and J.~Yong,
\it Sovling forward-backward stochastic differential equations explicitly --- a four step scheme,
\rm Probab. Theory Related Fields, {\bf 98} (1994), 339--359.

\bibitem{Ma2015} J. Ma et al.,
\it On well-posedness of forward-backward SDEs---a unified approach,
\rm Ann. Appl. Probab., {\bf 25} (2015),  2168--2214.

\bibitem{Peng-Wu1999} S. Peng and Z. Wu,
\it Fully coupled forward-backward stochastic differential equations and applications to optimal control,
\rm SIAM J. Control Optim.,  {\bf37} (1999),  825--843.


\bibitem{Viens-Zhang2019} F.~Viens and J.~Zhang,
\it A martingale approach for fractional Brownian motions and related path dependent PDEs,
\rm Ann. Appl. Probab., {\bf 29} (2019),  3489--3540.

\bibitem{Wang-Yong2021} H.~Wang and J.~Yong,
\it Time-inconsistent stochastic optimal control problems and backward stochastic Volterra integral equations,
\rm ESAIM Control Optim. Calc. Var., {\bf27}  (2021), 22.

\bibitem{Wang-Yong-Zhang2022} H.~Wang, J.~Yong, and J.~Zhang,
\it Path dependent Feynman--Kac formula for forward backward stochastic Volterra integral equations,
\rm  Ann. Inst. Henri Poincar\'{e} Probab. Stat.,  {\bf58} (2022), 603--638.

\bibitem{Wang-Yong-Zhou2023} H. Wang, J. Yong, and C. Zhou,
\it Linear-quadratic optimal controls for stochastic Volterra integral equations: causal state feedback and path-dependent Riccati equations,
\rm SIAM J. Control Optim., {\bf61} (2023),  2595--2629.

\bibitem{Yong1997} J.~Yong,
\it Finding adapted solutions of forward--backward stochastic differential equations:
method of continuation,
\rm Probab. Theory Related Fields, {\bf107} (1997), 537--572.

\bibitem{Yong2008} J.~Yong,
\it Well-posedness and regularity of backward stochastic Volterra integral equations,
\rm Probab. Theory Related Fields, {\bf142} (2008), 21--77.

\bibitem{Yong2010} J.~Yong,
\it Forward-backward stochastic differential equations with mixed initial-terminal conditions,
\rm Trans. Amer. Math. Soc., {\bf362} (2010),  1047--1096.

\bibitem{Yong-Zhou1999} J.~Yong and X.~Y.~Zhou,
\it Stochastic Controls: Hamiltonian Systems and HJB Equations,
\rm Springer-Verlag, New York, 1999.

\bibitem{Yu2022} Z. Yu,
\it On forward-backward stochastic differential equations in a domination-monotonicity framework,
\rm Appl. Math. Optim., {\bf85} (2022),  5.


\end{thebibliography}
\end{document}